\documentclass[
  a4paper, 
  reqno, 
  oneside, 
  12pt
]{amsart}

\usepackage[utf8]{inputenc}

\usepackage[dvipsnames]{xcolor} % Must come before tikz!
\colorlet{cite}{red}
\usepackage{tikz-cd}  % for diagrams
\usepackage{amsmath}

\usetikzlibrary{arrows,positioning}
\tikzset{ 
  baseline=-2.3pt,
  text height=1.5ex, text depth=0.25ex,
  >=stealth,
  node distance=2cm,
  mid/.style={fill=white,inner sep=2.5pt},
}

\usepackage{lmodern}
\usepackage{tikz-cd}
\usepackage{amsthm, amssymb, amsfonts}

\usepackage{amsthm, amssymb, amsfonts}	% Standard maths packages.
\usepackage[all]{xy}
\usepackage{graphicx,caption,subcaption}
\usepackage{braket}  % For nice sets
\usepackage{microtype}
\usepackage{DotArrow}
\usepackage[makeroom]{cancel}
\usepackage[%
  bookmarks=true,			% show bookmarks bar?
  unicode=true,			% non-Latin characters in Acrobat’s bookmarks
  pdftitle={0-shifted cosymplectic},		%
  pdfauthor={Lopez-Garcia, Valencia},	% 
  pdfkeywords={precosymplectic, 0-shifted structures, reduction},	% list of keywords
  colorlinks=true,		% false: boxed links; true: colored links
  linkcolor=black,			% color of internal links
  citecolor=black,		% color of links to bibliography
  filecolor=magenta,		% color of file links
  urlcolor=RoyalBlue			% color of external links
]{hyperref}				% One of the most useful packages I have found.
\usepackage{cleveref}

\newtheoremstyle{mydef}
  {}		% Space above environment
  {}		% Space below environment
  {}		% Body font
  {}		% Indent amount (empty = no indent, \parindent = para indent)
  {\scshape}	% theorem head font
  {. }		% Punctuation after heading
  { }		% Space after heading
  {\thmname{#1}\thmnumber{ #2}\thmnote{ #3}}	% Heading spec

\newtheorem{theorem}{Theorem}[section]
\newtheorem*{theorem*}{Theorem}
\newtheorem{proposition}[theorem]{Proposition}
\newtheorem*{proposition*}{Proposition}
\newtheorem{lemma}[theorem]{Lemma}
\newtheorem*{lemma*}{Lemma}
\newtheorem{corollary}[theorem]{Corollary}
\newtheorem*{corollary*}{Corollary}
\theoremstyle{definition}
\newtheorem{definition}[theorem]{Definition}
\newtheorem{example}[theorem]{Example}

\theoremstyle{remark}

\newtheorem*{conjecture*}{Conjecture}
\usepackage{tikz}

\DeclareMathOperator{\im}{im}

\newcommand{\rr}{\rightrightarrows}

\author{Daniel L\'opez-Garcia {\tiny and } Fabricio Valencia}
\subjclass[2020]{53D20, 57R30, 58H05}
\address{}
\date{\today}

\address{D. L\'opez-Garcia - Instituto de Matem\'atica e Estat\'istica, Universidade Federal Fluminense,  Rua Prof. Marcos Waldemar de Freitas Reis s/n, 24210-201 Niter\'oi - Brazil.  \newline  
      \phantom{xx}
       F. Valencia - Instituto de Matem\'atica e Estat\'istica, Universidade de S\~ao Paulo, Rua do Mat\~ao 1010, Cidade Universit\'aria, 05508-090 S\~ao Paulo - Brazil.
      \newline  
      \phantom{xx}
  danielflg@id.uff.br, fabricio.valencia@ime.usp.br}

\title{Hamiltonian actions on $0$-shifted cosymplectic groupoids}

\begin{document}
\maketitle

\begin{abstract}
We introduce the notion of 0-shifted cosymplectic structure on differentiable stacks and develop a corresponding moment map theory for Hamiltonian cosymplectic actions. We present a reduction procedure, establish a version of the Kirwan convexity theorem, and obtain examples of Morse–Bott Lie groupoid morphisms.
\end{abstract}

%\tableofcontents
\section{Introduction}

Cosymplectic geometry can be thought of as an odd counterpart of symplectic geometry, yielding a different approach to that provided by contact geometry \cite{Libermann1,Libermann2}. One of the main advantages of working with cosymplectic structures is that they allow one to deal with time-dependent Hamiltonian systems which can not be directly approached using standard symplectic structures \cite{GuzmanMarrero}. In such a context, the geometry is described via a manifold $M$ endowed with a closed differential two-form $\omega$ and a non-vanishing closed differential one-form $\eta$, so that $TM=\ker(\omega)\oplus \ker(\eta)$. This implies that $M$ is both $(2n+1)$-dimensional and oriented since $\eta\wedge \omega^n$ turns out to be a volume form. There are situations in which one can address numerous time-dependent Hamiltonian systems after weaken the condition for the pair $(\omega, \eta)$ to determine a cosymplectic structure. Instead, we may require that $\ker(\omega)\cap \ker(\eta)$ is a regular distribution which is strictly contained in $\ker(\omega)$ and therefore it gives rise to a foliation $\mathcal{F}_{(\omega, \eta)}$ of $M$, see \cite{ChineaDeLeonMarrero}. The latter immediately implies that $(M,\omega)$ is a presymplectic manifold. As expected, the space of leaves $M/\mathcal{F}_{(\omega, \eta)}$ is in general a singular space which one can desingularize by considering a foliation groupoid integrating the Lie algebroid associated with the tangent distribution $T\mathcal{F}_{(\omega, \eta)}$. In this regard, it is generally accepted that the holonomy and the
monodromy groupoids of the foliation are actually extreme examples of the possible integrations that we can consider \cite{CrainicMoerdijk}.

Building upon the recent works \cite{HoffmanSjamaar,LinSjamaar,MaglioTortorellaVitagliano}, we introduce the notion of 0-shifted cosymplectic groupoid which plays out the role of being the global geometric object integrating the foliation associated with a precosymplectic structure. Such a notion is Morita invariant, so that it enables us to define 0-shifted cosymplectic structures on differentiable stacks. Additionally, we introduce Hamiltonian actions of Lie groups on precosymplectic manifolds and describe their main properties. This can be done by following the notion of Hamiltonian action as well as moment map in the realm of cosymplectic geometry which has been widely explored in the existing literature \cite{Albert,BazzoniGoertsches,BazzoniGoertsches2,deLucasRivasVilarinoZawora}. As an immediate consequence, we then study some aspects concerning a notion of Hamiltonian action of foliation Lie 2-groups on Lie groupoids endowed with 0-shifted cosymplectic structures, pointing out that such a study can be accomplished by applying several key results of the general theory developed in \cite{HoffmanSjamaar}. In particular, we present a corresponding Marsden--Weinstein--Meyer reduction procedure for 0-shifted cosymplectic groupoids, establish a version of the Kirwan convexity theorem, and explain why the function components of our Moment map morphisms yield examples of Morse–Bott Lie groupoid morphisms in the sense of \cite{OrtizValencia}. We also provide several examples illustrating the main results described above. It is worth mentioning that such a reduction procedure gives rise to new insights even in the setting of cosymplectic reduction by an ordinary Lie group. Besides, all these features can be used to provide a sort of Delzant's classification for what we would call a toric cosymplectic stack. In fact, this can be done by following the works \cite{BazzoniGoertsches2,Hoffman} in order to classify simple convex polytopes equipped with some additional combinatorial data.

We stress that it was recently defined a notion of cosymplectic groupoid \cite{FernandesIglesias} along with the concept of Hamiltonian action by a Lie group \cite{LopezGarciaMartinez}. Nevertheless, the main difference between the two notions lies in the fact that cosymplectic groupoids are defined in terms of multiplicative cosymplectic structures on the arrows, whereas 0-shifted cosymplectic structures are defined in terms of basic precosymplectic structures on the objects. More importantly, both notions motivate the problem of studying general $m$-shifted cosymplectic structures on Lie $n$-groupoids, something that can be accomplished by adapting the general theory developed in \cite{CuecaZhu} for the symplectic geometry case.

\vspace{.2cm}
{\bf Acknowledgments:} L\'opez-Garcia was supported by Grant 2022/04705-8 Sao Paulo Research Foundation - FAPESP. Valencia was supported by Grant 2024/14883-6 Sao Paulo Research Foundation - FAPESP.

\section{Precosymplectic Hamiltonian actions}

In this section we deal with some features about a natural notion of Hamiltonian action in the category of precosymplectic manifolds. A \emph{cosymplectic manifold} is a triple $(M,\omega,\eta)$ where $M$ is a smooth manifold, $\omega\in \Omega^2(M)$ is a closed differential 2-form, and $\eta\in \Omega^1(M)$ is a non-vanishing closed differential 1-form such that $\ker(\omega)\oplus \ker(\eta)=TM$. Equivalently, $M$ is a $(2n+1)$-dimensional manifold such that $\eta\wedge \omega^n$ defines a volume form on $M$. The pair $(\omega,\eta)$ induces a vector bundle isomorphism $\flat:TM \to T^\ast M$ which is defined by $\flat(v)=\iota_v\omega +\eta(v)\eta$. We refer to $\flat$ as the \emph{Lichnerowicz map} associated with $(\omega,\eta)$. The  \emph{Reeb vector field} of $(M,\omega,\eta)$ is defined by  $v=\flat^{-1}(\eta)$ and it is characterized by the conditions $\iota_v\omega=0$ and $\iota_v\eta=1$.  It is clear that the Lichnerowicz map $\flat$ together with a Reeb-like vector field completely characterize any cosymplectic structure \cite{Albert}.

Interesting examples of  cosymplectic manifolds can be provided by using the symplectic mapping tori construction \cite{BazzoniGoertsches2,FernandesIglesias,Li}. Namely, let $(S,\omega_S)$ be a symplectic manifold and $\varphi:S\to S$ be a symplectomorphism. Consider the quotient of $S\times \mathbb{R}$ by the action of $\mathbb{Z}$ given by $n\cdot (x,r)=(\varphi^n(x),r+n)$. Such an action is free and proper, so that the corresponding quotient space is a smooth manifold which we denote by $S_\varphi$. Additionally, there is a fiber bundle $S\to S_\varphi\xrightarrow{q} S^1$. The cosymplectic structure $(\omega,\eta)$ on $S_\varphi$ is defined by projecting the basic form $\textnormal{pr}_S^\ast (\omega_S)$ onto $S_\varphi$ and setting $\eta=q^\ast (d\theta)$, where $\theta$ stands for the coordinate direction in $S^1$. It follows that the latter construction is a one-to-one correspondence if the symplectic manifold we start with is closed \cite{Li}.

We are interested in briefly describing some features of the case in which a pair $(\omega,\eta)$ not necessarily determines a cosymplectic structure on $M$, yet satisfying a regularity condition in terms of its Lichnerowicz map. More precisely, let $M$ be a smooth manifold equipped with a closed differential 2-form $\omega$ and a closed differential 1-form $\eta$. We say that the triple $(M,\omega,\eta)$ is a \emph{precosymplectic manifold} if $\ker(\omega)\cap \ker(\eta)$ gives rise to a regular distribution which is strictly contained in $\ker(\omega)$, see \cite{ChineaDeLeonMarrero,GraciaDeLucasRivasRoman-Roy}. The latter implies that $\eta\wedge \omega^r$ is a non-vanishing form and $\omega^{r+1} = 0$ for a certain fixed $r$. In other words, $\omega$ has constant rank $2r< \dim(M)$ and the pair $(M,\omega)$ becomes a presymplectic manifold \cite{LinSjamaar}.

The following key observation will be fundamental later on. 

\begin{lemma}\label{Lemma1}
If $(M,\omega,\eta)$ is a precosymplectic manifold then $\ker(\flat)=\ker(\omega)\cap \ker(\eta)$.
\end{lemma}
\begin{proof}
It is simple to verify that $\ker(\omega)\cap \ker(\eta) \subseteq \ker(\flat)$. Conversely, using the formula
$$(r+1)\flat(v)\wedge \omega^r/(r+1)=\iota_v(\omega^{r+1})/(r+1)+\eta(v)\eta\wedge \omega^r,$$
we get that if $v\in\ker(\flat)$ then $\eta(v)=0$ and consequently $\iota_v\omega=-\eta(v)\eta=0$.
\end{proof}

The previous result guarantees that the Lichnerowicz map of a precosymplectic $(M,\omega,\eta)$ has constant rank and its kernel determines regular involutive distribution on $M$, as $\ker(\omega)$ and $\ker(\eta)$ are so. We shall denote by $\mathcal{F}_\flat$ the foliation of $M$ corresponding to $\ker(\flat)$. Another important foliation showing up into the whole picture is the null foliation $\mathcal{F}_\omega$ associated with the regular involutive distribution $\ker(\omega)$.

Let $\mathfrak{X}_\textnormal{t}(M,\mathcal{F}_\flat)$ and $\Omega^1(M,\mathcal{F}_\flat)$ denote the space of transverse vector fields and basic differential 1-forms with respect to $(M,\mathcal{F}_\flat)$, respectively. The Lichnerowicz map descends to a linear isomorphism $\flat:\mathfrak{X}_\textnormal{t}(M,\mathcal{F}_\flat)\to \Omega^1(M,\mathcal{F}_\flat)$. It is clear that $\eta \in \Omega^1(M,\mathcal{F}_\flat)$ so there exists a unique $v\in \mathfrak{X}_\textnormal{t}(M,\mathcal{F}_\flat)$ such that $\flat(v)=\eta$. We refer to $v$ as the \emph{transverse Reeb vector field} associated with $\eta$. This allows us to define a Poisson structure on the space of basic functions on $M$ with respect to the foliation $\mathcal{F}_\flat$ (i.e. smooth functions constant along the leaves) by setting
$$ \lbrace f_1,f_2\rbrace = \omega(\overline{v}_1,\overline{v}_2)=\mathcal{L}_{v_2}(f_1),$$
where $\overline{v}_1$ and $\overline{v}_2$ are representative of the elements $v_1$ and $v_2$ in $\mathfrak{X}_\textnormal{t}(M,\mathcal{F}_\flat)$ verifying $\flat(v_1)=df_1$ and $\flat(v_2)=df_2$ for basic functions $f_1$ and $f_2$, respectively, compare \cite{Bursztyn,LinSjamaar}. Such an expression is well-defined since $\ker(\flat)=\ker(\omega)\cap \ker(\eta)$.

The notion of Hamiltonian action as well as moment map in the realm of cosymplectic geometry has been widely explored in the existing literature. Consult for instance \cite{Albert,BazzoniGoertsches,BazzoniGoertsches2,deLucasRivasVilarinoZawora} and their quoted references. Building upon what happens in the case of presymplectic geometry we can introduce a notion of Hamiltonian action for precosymplectic structures.

\begin{definition}\label{def:HamiltonianPrecocymplectic}
Let $(M,\omega,\eta)$ be a precosymplectic manifold and $K$ denote a Lie group with Lie algebra $\mathfrak{k}$. We say that an action of $K$ on $M$ is \emph{precosymplectic} if $k^\ast \omega=\omega$ and $k^\ast \eta=\eta$ for all $k\in K$. Accordingly, the action is said to be \emph{Hamiltonian} if there exists a smooth map $\mu:M\to \mathfrak{k}^\ast$ called \emph{moment map} such that
\begin{itemize}
\item for each $\xi\in \mathfrak{k}$ we have that $\eta(\xi_M)=0$ and $d\mu^\xi=\iota_{\xi_M}\omega$, and
\item $\mu$ is $\textnormal{Ad}^\ast$-equivariant.
\end{itemize}
\end{definition}

In consequence, a precosymplectic moment map is in particular presymplectic moment map. Let us illustrate the notions above with some examples. The fist one is motivated by the main example in \cite{HoffmanSjamaar}.

\begin{example}
	\label{Ex:Precosymplecticfromcosymplectic}
	Let $(M,\omega,\eta)$ be a cosymplectic manifold admitting a Hamiltonian action of the compact torus $K=\mathbb{T}^n$ whose Lie algebra $\mathfrak{k}=\mathbb{R}^n$ is abelian with moment map $\mu:M\to \mathfrak{k}^\ast$, see \cite{BazzoniGoertsches2}. Consider any integer $0\leq k< n$ and the normal subgroup $N=\mathbb{T}^k\subset K$ with Lie algebra $\mathfrak{n}=\mathbb{R}^k$. Let $i:\mathfrak{n}\hookrightarrow \mathfrak{k}$ be the inclusion and $i^\ast:\mathfrak{k}^\ast \to \mathfrak{n}^\ast$ denote its dual map. If $0\in \mathfrak{n}^*$ is a regular value of $i^\ast \circ \mu$ then the manifold $X_0=(i^*\circ \mu)^{-1}(0)$ inherits a precosymplectic structure given by $\omega_0=\omega|_{X_0}$ and $\eta_0=\eta|_{X_0}$. Furthermore, $(X_0,\omega_0,\eta_0)$ gives rise to a precosymplectic Hamiltonian $K$-manifold with moment map $\mu|_{X_0}:X_0\to \mathfrak{n}^\circ \cong (\mathfrak{k}/\mathfrak{n})^\ast$.
\end{example}

\begin{example}
	\label{Ex:CosymplecticTrivialProduct}
	Let $(S,\mathbb{T}^n,\omega_S)$ be a toric symplectic manifold with moment map $\tilde \mu:S\to \mathfrak{k}^\ast$ and let $S_{\textnormal{id}}$ be the cosymplectic manifold defined by the symplectic mapping tori construction with the identity map, so that $S_{\textnormal{id}}=S\times S^1$. The construction described in Example \ref{Ex:Precosymplecticfromcosymplectic} yields a trivial extension of the toric symplectic manifold $(S,\mathbb{T}^n,\mu)$ to $S\times S^1$. In other words, the new moment map $\mu:S\times S^1\to \mathfrak{k}^\ast$ is determined by trivially extending $\tilde{\mu}$. Let $Y_0=(i^*\circ \tilde \mu)^{-1}(0)$. The foliation $\mathcal{F}_\flat$ is obtained from the distribution $\ker(\omega_S|_{Y_0})$ by setting each leaf $L_{(x,\theta)}\subset Y_0\times S^1$ as the leaf $L_x\subset Y_0$, while the foliation $\mathcal{F}_\omega$ is defined by considering the extension $L_{(x,\theta)}=L_x\times S^1\subset Y_0\times S^1$. 
	
	The behavior of the foliations $\mathcal{F}_\flat$ and  $\mathcal{F}_\omega$ might be much more complicated for a general $\mathbb{T}^n$-equivariant symplectomorphism $\varphi: S\to S$ as the following particular instance shows. Consider  $(\mathbb{C}^2, dz_1\wedge d\bar z_1+dz_2\wedge d\bar z_2)$ with $\mathbb{T}^2$-action  on $\mathbb{C}^2$ by $(\theta_1,\theta_2)\cdot(z_1,z_2)=(e^{i\theta_1}z_1,e^{i\theta_2}z_2)$. We have a moment map $\tilde\mu:\mathbb{C}^2\to (\mathbb{R}^2)^\ast\cong\mathbb{R}^2 $ defined by
	$$\tilde\mu(z_1,z_2)=(|z_1|^2-1,|z_2|^2-1).$$

	If $N=S^1\times \{1\}\subset \mathbb{T}^2$ then $Y_0=\{(z_1,z_1)\in \mathbb{C}^2: |z_1|=1\}\cong S^1\times \mathbb{C}$. We can get concrete information about the foliations $\mathcal{F}_\omega$ and $\mathcal{F}_\flat$ depending on the symplectomorphism $\varphi:\mathbb{C}^2\to \mathbb{C}^2$ we choose. For instance:
	\begin{itemize}
		\item Let $\varphi$ be the identity map $\textnormal{id}_{\mathbb{C}^2}$. In this case $Y_0=S^1\times \mathbb{C}\times S^1\cong \mathbb{T}^2\times \mathbb{C}$, so that the foliation $\mathcal{F}_\omega$ is given by the fibers of the projection onto $\mathbb{C}$, whereas the foliation $\mathcal{F}_\flat$ is the sub-foliation of $\mathcal{F}_\omega$ given by the circles $S^1$ in $\mathbb{T}^2$. In both cases, the holonomies are trivial.
		
		\item Let $\varphi(z_1, z_2)=(z_1,-z_2)$, so as to $Y_0=S^1\times (\mathbb{C}\times \mathbb{R})/\mathbb{Z}$, where the quotient space is obtained by identifying $(z_2,0)\sim (-z_2,1)$. In this scenario, the foliation $\mathcal{F}_\omega$ is given by fixing the second coordinate. Observe that not all the holonomies of $\mathcal{F}_\omega$ are trivial. Consider the leaf $L_{0}$ passing through the point $(z_1,[0,r])$. Let $\gamma:I\to L_0$ be the loop given by $\gamma(t)=(z_1,[0,e^{2\pi i t}])$. If we pick the transversal section $\Sigma$ parametrized by  $z\in \mathbb{C}$ with $\vert z\vert$ small enough, then each $\gamma_{z}(t)=(z_1,[z,e^{2\pi i t}])$ corresponds to a lifted path over $\gamma(t)$. Hence, every path $\gamma_{z}(t)$ intersects $\Sigma$ at the values associated with $z$ and $-z$. Therefore, the holonomy of $L_0$ is $\mathbb{Z}_2$. Analogously, one can construct examples with holonomy groups $\mathbb{Z}_p$ using a rational rotation or with holonomy group $\mathbb{Z}$ using an irrational rotation.
	\end{itemize}	
\end{example}

Let us step back to the general setting. From Definition \ref{def:HamiltonianPrecocymplectic} it follows that $(M,\omega)$ is a presymplectic Hamiltonian $K$-manifold \cite{LinSjamaar}. Let us define $\mathfrak{n}=\lbrace \xi\in \mathfrak{k}: \flat(\xi_M)=0\rbrace$. Observe that $\flat(\xi_M)=\iota_{\xi_M}\omega$ and therefore $\mathfrak{n}=\lbrace \xi\in \mathfrak{k}: \iota_{\xi_M}\omega=0\rbrace$. It is simple to verify that $\mathfrak{n}$ is an ideal of $\mathfrak{k}$. Define $N$ to be the normal connected immersed (but not necessarily closed) Lie subgroup of $K$ whose Lie algebra is $\mathfrak{n}$. The Lie group $N$ acts trivially on the leaf spaces $M/\mathcal{F}_\flat$ and $M/\mathcal{F}_\omega$, so the induced action of $K$ descends to topological action of the quotient group $K/N$ along such leaf spaces. If $x\in M$, $v\in T_x \mathcal{F}_\flat=\ker(\omega_x)\cap \ker(\eta_x)$, and $\xi\in \mathfrak{k}$ then
$$d\mu^\xi_x(v)=\omega_x(\xi_M(x),v)=-\omega_x(v,\xi_M(x))-\eta_v(v)\eta_x(\xi_M(x))=-\flat_x(v)(\xi_{M}(x))=0.$$

This implies that the moment map $\mu$ also descends leaf spaces $M/\mathcal{F}_\flat$ and $M/\mathcal{F}_\omega$. Additionally, if $M$ is connected then the affine span of the image $\mu(M)$ is of the form
$\lambda+\mathfrak{n}^\circ$ for some element $\lambda\in \mathfrak{k}^\ast$ which is fixed under the coadjoint action of $K$ on $\mathfrak{k}^\ast$. The latter assertion can be similarly shown as in \cite[Prop. 2.9.1]{LinSjamaar} for the presymplectic case. In this scenario, we can replace $\mu$ with $\mu-\lambda$ to obtain a new
$\textnormal{Ad}^\ast$-equivariant moment map which maps $M$ into $\mathfrak{n}^\circ\cong (\mathfrak{k}/\mathfrak{n})^\ast$ and which
descends to a continuous map $\mu_\flat: M/\mathcal{F}_\flat\to (\mathfrak{k}/\mathfrak{n})^\ast$. 

Following \cite{LinSjamaar}, we say that the action of $K$ on $(M,\omega)$ is \emph{clean} if $T_x(N\cdot x)=T_x(K\cdot x)\cap T_x \mathcal{F}_\omega$ for all $x\in M$. Since $\eta(\xi_M)=0$ for each $\xi\in \mathfrak{k}$ we obtain that $T_x(K\cdot x)\subseteq \ker(\eta)$. Thus, by Lemma \ref{Lemma1} it follows that such a clean action requirement amounts to asking
$$T_x(N\cdot x)=T_x(K\cdot x)\cap T_x \mathcal{F}_\flat,\quad x\in M.$$

Hence, as a direct application of the main results proven in \cite{LinSjamaar}, one can get the following expected analogues in the precosymplectic context. This stems from the fact that a precosymplectic Hamiltonian action is in particular a presymplectic Hamiltonian action.

\begin{proposition}\label{prop:MomentBodyMorseBott}
Let $(M,\omega,\eta)$ be a precosymplectic manifold which admits a clean Hamiltonian action of a compact and connected Lie group $K$ with moment map $\mu: M\to \mathfrak{k}^\ast$.
\begin{itemize}
\item Suppose that $M$ is connected and that $\mu:M\to \mathfrak{k}^\ast$ is proper. Choose a maximal torus $T$ of $K$ and a closed Weyl chamber $C$ in $\mathfrak{t}^\ast$, where $\mathfrak{t}$ is the Lie algebra of $T$, and define $\triangle(M) =\mu(M)\cap C$. Then, $\triangle(M)$ is a closed convex polyhedral set. It is rational if and only if $N$ is closed.
\item If $K$ is a torus then for every $\xi\in\mathfrak{k}$ the component function $\mu^\xi$ of the moment map $\mu$ is a Morse-Bott function and the index of every component of $\textnormal{Crit}(\mu^\xi)$ is even.
\end{itemize}
\end{proposition}

Using Example \ref{Ex:CosymplecticTrivialProduct} we can illustrate the previous result in a very simple situation.

\begin{example} \label{Ex:Complexspace}
Let us consider the toric symplectic manifold $(\mathbb{C}^n, \mathbb{T}^n, \sum_{j=1}^ndz_j\wedge d\bar z_j)$ with moment map $\tilde \mu:\mathbb{C}^n\to (\mathbb{R}^n)^\ast\cong \mathbb{R}^n$ defined by $\tilde{\mu}(z_1,\ldots,z_n)=\sum_{j=1}^n(|z_j|^2-1)\xi_j$. Here $\lbrace \xi_j\rbrace$ stands for the canonical basis of $\mathbb{R}^n$. If $\mathfrak{n}=\langle \xi_1,\ldots, \xi_k\rangle$ then

 $$X_0=\{((z_1,\ldots,z_n),\alpha)\in \mathbb{C}^n\times S^1: \vert z_j\vert =1, \textnormal{ for } 1 \leq j\leq k\}.$$

The torus $\mathbb{T}^n$ acts on $X_0$ by $(\theta_1,\ldots\theta_n)\cdot ((z_1,\ldots, z_n),\theta)=((e^{i\theta_1}z_1,\ldots,e^{i\theta_n}z_n),\theta)$. Such an action preserves the precosymplectic structure on $X_0$ and inherits a moment map $\mu:X_0\to (\mathbb{R}^n/\mathfrak{n})^*\cong \mathbb{R}^{n-k}$ which is given by $\mu(z,\theta)=\sum_{j=k+1}^n(|z_j|^2-1)\xi_j$. Note that 

$$\mu(X_0)=\bigl\{\sum_{j=k+1}^n r_{j}\xi_j:\ r_j\geq -1\bigr\}.$$

We may depict the image of $\mu$ for $n=3$ and $k=1$ as in Figure \ref{Fig:Polyhedron}.

\begin{figure}[h!]
	\centering
	\includegraphics[width=0.60\textwidth]{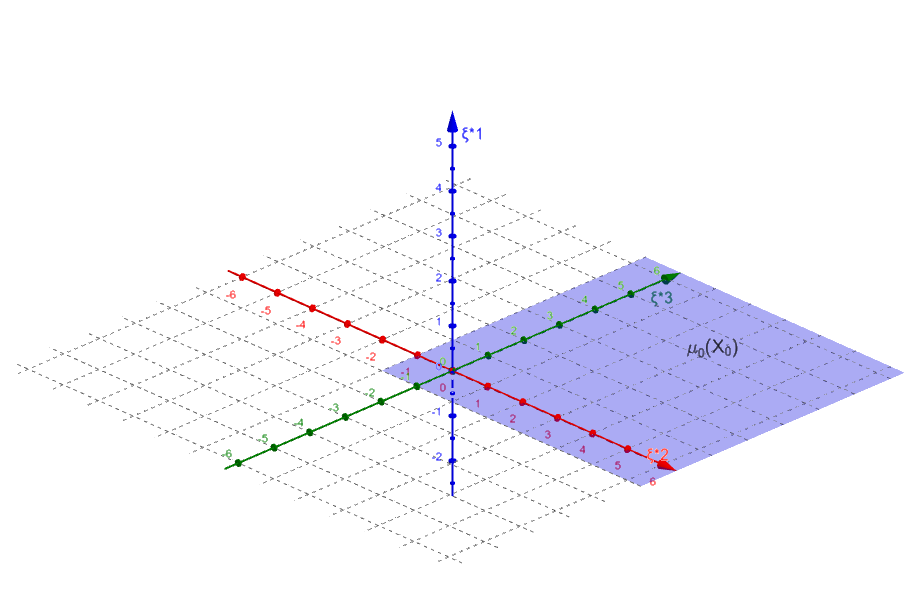}
	\caption{\footnotesize $\mu(X_0)$ for $S=\mathbb{C}^3$ and $\mathfrak{n}=\langle \xi_1 \rangle$.}
	\label{Fig:Polyhedron}
\end{figure}
\end{example}

In the next sections we shall discover that these results for Hamiltonian precosymplectic structures can be easily promoted to Hamiltonian $0$-shifted cosymplectic versions over differentiable stacks. 

\section{0-shifted cosymplectic structures}

Throughout this section we assume that the reader is familiar with the notion of Lie groupoid and the geometric/topological aspects underlying its structure \cite{delHoyo}. The structural maps of a Lie groupoid $G\rr M$ are denoted by $(s,t,m,u,i)$ where $s,t:G\to M$ are the maps respectively indicating the source and target of the arrows, $m:G^{(2)}\to G$ stands for the partial composition of arrows, $u:M\to G$ is the unit map, and $i:G\to G$ is the map determined by the inversion of arrows. The Lie algebroid of $G$ is determined by the vector bundle $A=\ker(ds)|_M$ with anchor map  given by $\rho=dt|_A$. A differential for $\omega\in \Omega^\bullet(M)$ is said to be \emph{basic} if $(t^\ast-s^\ast)(\omega)=0$. If $G$ is source-connected, the condition for $\omega$ being basic is equivalent to asking that it is both \emph{horizontal} $\iota_{\rho(a)}\omega=0$ and \emph{invariant} $\mathcal{L}_{\rho(a)}\omega=0$ for all $a\in \Gamma(A)$. Although basic differential forms are defined on $M$ we denote the space of basic forms by $\Omega_{\textnormal{bas}}^\bullet(G)$, hoping that there is no risk of confusion.

Motivated by the notion of $0$-shifted symplectic structure studied in \cite{HoffmanSjamaar} we set up the following definition.

\begin{definition}\label{def:0shifted}
A $0$-\emph{shifted cosymplectic structure} on $G\rr M$ is determined by a pair of closed basic differential forms $\omega\in \Omega_{\textnormal{bas}}^2(G)$ and $\eta\in \Omega_{\textnormal{bas}}^1(G)$ with $\eta$ nowhere vanishing such that the corresponding Lichnerowicz map $\flat$ induces a quasi-isomorphism for all $x\in M$:
\begin{equation}\label{diag:cochain}
	\begin{tikzcd}
		0 \arrow[r] & A_x \arrow[r, "\rho_x"] \arrow[d] & T_xM \arrow[r] \arrow[d, "\flat_x"] & 0 \arrow[r]  \arrow[d] & 0 \\
		0 \arrow[r] & 0 \arrow[r] & T^{\ast}_xM \arrow[r, "\rho^{\ast}_x"'] & A^{\ast}_x \arrow[r] & 0
	\end{tikzcd}
\end{equation}
\end{definition}

A couple of observations come in order as follows. First of all, the fact that $\flat$ indeed defines a cochain map between the complexes depicted above follows from $\omega$ and $\eta$ being basic differential forms. For instance, to verify that $\flat\rho=0$ we pick $a\in A_x$ and $v\in T_xM$. There is $\tilde{v}\in T_{1_x} G$ such that $dt_{1_x}(\tilde{v})=v$, so that:
\begin{eqnarray*}
	\langle \flat(\rho(a)),v \rangle_x &=& \omega_x(\rho_x(a),v)+\eta_x(\rho_x(a))\eta_x(v)=(t^*\omega)_{1_x}(a, \tilde v)+(t^*\eta)_{1_x}(a)\eta_x(v)\\
	&=& (s^*\omega)_{1_x}(a, \tilde v)+(s^*\eta)_{1_x}(a)\eta_x(v)=0.
\end{eqnarray*}

Similarly, $\rho^*\flat=0$. Second, to ask for $\flat$ being a quasi-isomorphism for all $x\in M$ is equivalent to asking that the anchor map $\rho:A\to TM$ is injective, which in turn is equivalent to having that $G\rr M$ is a foliation groupoid (see \cite{CrainicMoerdijk}), and $\ker(\flat)=\im(\rho)$. Therefore, $\flat$ yields a fiberwise isomorphism between $TM/\im(\rho)$ and $\ker(\rho^\ast)$. Observe that if $\ker(\flat)$ is strictly contained in $\ker(\omega)$ then we get that $(M,\omega,\eta)$ becomes a precosymplectic manifold. Additionally, if $\tilde{\omega}=s^\ast \omega=t^\ast \omega$ and $\tilde{\eta}=s^\ast \eta=t^\ast \eta$ then $(G,\tilde{\omega},\tilde{\eta})$ is also a precosymplectic manifold with $\ker(\tilde{\flat})=\ker(ds)+\ker(dt)$. 

Let us consider the normal vector bundles $N_0=TM/\ker(\flat)$ and $N_1=TG/(\ker(ds)+\ker(dt))$. Similarly as proven in \cite{HoffmanSjamaar}, there are isomorphisms $N_1\cong s^\ast N_0$ and $N_1\cong t^\ast N_0$ which allow one to obtain well-defined pullbacks $s^\ast, t^\ast:\Gamma(N_0)\to \Gamma(N_1)$. In particular, the space of basic vector fields $\Gamma_{\textnormal{bas}}(G):=\lbrace v\in \Gamma(N_0): s^\ast v= t^\ast v\rbrace$ inherits a Lie algebra structure which models the Lie algebra of vector fields of the differentiable stack $[M/G]$ presented by $G\rr M$. The fact the Lichnerowicz map $\flat:TM\to T^\ast M$ induces a quasi-isomorphism implies that it descends to a linear isomorphism 
$$\flat:\Gamma_{\textnormal{bas}}(G)\to \Omega_{\textnormal{bas}}^1(G),$$
defined by $\flat(v)=\iota_{\overline{v}}\omega+\eta(\overline{v})\eta$ where $\overline{v}\in \mathfrak{X}(M)$ is a representative of $v\in \Gamma_{\textnormal{bas}}(G)$. Such an expression is well-defined, as it is independent of the choice of representatives, compare \cite[p. 15]{HoffmanSjamaar}. The latter enables us to introduce the \emph{Reeb vector field} of the pair $(\omega,\eta)$ as the basic vector field $v$ such that $\flat(v)=\eta$. Thus, the space of basic functions $C^\infty(M)^G=\Omega_{\textnormal{bas}}^0(G)$ admits a Poisson structure induced by $(\omega,\eta)$ which is defined by  $\lbrace f_1,f_2\rbrace=\omega(\overline{v}_1,\overline{v}_2)=\mathcal{L}_{v_2}(f_1)$, where $\overline{v}_1,\overline{v}_2\in \mathfrak{X}(M)$ are representatives of the basic vector fields $v_1,v_2\in \Gamma_{\textnormal{bas}}(G)$ satisfying $\flat(v_1)=df_1$ and $\flat(v_2)=df_2$ for all $f_1,f_2\in C^\infty(M)^G$. Recall that such an expression is well-defined since $\ker(\flat)=\ker(\omega)\cap \ker(\eta)$. 

By mimicking the approach adopted in \cite{Maglio} to deal with 0-shifted symplectic structures it is simple to check that the notion of 0-shifted cosymplectic structure introduced in Definition \ref{def:0shifted} is Morita invariant. This can be summarized in the following result.

\begin{proposition}\label{Prop: Morita invariance}
	Let $G\rr M$ be a Lie groupoid equipped with a pair of closed basic forms $\omega\in \Omega_{\textnormal{bas}}^2(G)$ and $\eta\in \Omega_{\textnormal{bas}}^1(G)$. The following assertions hold true:
	\begin{enumerate}
		\item If the cochain map $\flat_{(\omega,\eta)}$ in \eqref{diag:cochain} is a quasi-isomorphism at $x\in M$ then it is a quasi-isomorphism at every point lying inside the groupoid orbit through $x$.
		\item Let $\phi:(G'\rr M')\to (G\rr M)$ be a Morita map\footnote{It is well-known that we can consider the Morita map to cover a surjective submersion on objects.}. Then, the cochain map $\flat_{(\omega,\eta)}$ is a quasi-isomorphism at all points in $M$ if and only if the cochain map $\flat_{(\phi^\ast \omega,\phi^\ast \eta)}$ is so at all points in $M'$. In particular, the pullback $\phi^\ast: \Omega_{\textnormal{bas}}^\bullet(G)\to \Omega_{\textnormal{bas}}^\bullet(G')$ gives rise to a one-to-one correspondence between 0-shifted cosymplectic structures on $G$ and $G'$.
	\end{enumerate}
\end{proposition}
\begin{proof}
These facts can be proven by arguing exactly as in Lemma 3.2.2 and  Propositions 3.2.3 and 3.2.4 from \cite{Maglio} since the differential forms we are dealing with are basic. Nevertheless, here we sketch a few details about the second assertion for the reader's convenience. Since $\phi$ is a Morita map, it follows that the chain map
\begin{equation*}
\begin{tikzcd}[row sep=2em, column sep=2em]
0\arrow[r]&A'_x\arrow[r,"\rho'"]\arrow[d,"d\phi"]&T_xM'\arrow[d,"d\phi"]\arrow[r]&0\\
0\arrow[r]&A_{\phi(x)}\arrow[r,"\rho"]& T_{\phi(x)}M\arrow[r]&0
\end{tikzcd}
\end{equation*}	
is a quasi-isomorphism for all $x\in M'$. More explicitly:

$$\ker(\rho')=\text{Lie}(G'(x,x))\simeq \text{Lie}(G(\phi(x),\phi(x)))=\ker(\rho),$$
and $$\text{coker}(\rho')\simeq N_xM'\simeq N_{\phi(x)}M\simeq \text{coker}(\rho).$$

Therefore, the Lie groupoid $G'$ is a foliation groupoid if and only if so is $G$. To verify our remaining quasi-isomorphism condition observe that $$\langle \phi^*\flat_{x}(u),v\rangle=\phi ^*\flat_{x}(u,v)=\flat_{\phi(x)}(d\phi(u),d\phi(v))=\langle \flat_{\phi(x)}(d\phi(u)),d\phi(v)\rangle,$$
for $u,v\in T_x M'$.
If we assume that $\text{im}(\rho_{\phi(x)})=\ker(\flat_{\phi(x)})$ and $u\in \ker(\phi^*\flat_x)$, since $\phi$ can be taken as a submersion, then $d\phi(u)\in \ker (\flat_{\phi(x)})$ and  $[d\phi(u)]$ is zero in $\text{coker}(\rho)$. Hence,  $[u]$ is zero in $\text{coker}(\rho')$. Thus,  $\ker(\phi^*\flat_x)=\text{im}(\rho'_x)$. Similarly, if we suppose that $\text{im}(\rho'_x)= \ker (\phi^*\flat_x)$ and $\tilde u\in \ker (\flat_{\phi(x)})$, then  $[\tilde u]$ is zero in $\text{coker}(\rho)$. That is to say, $\ker(\flat_{\phi(x)})=\text{im}(\rho_{\phi(x)}).$

From the above we get that if $\flat$ defines a quasi-isomorphism for all point in $M$, then so does $\phi^*\flat$ for all points in $M'$. Conversely, suppose that $\phi^*\flat_x$ defines a quasi-isomorphism for every point $x \in M'$. In this case, given $y \in M$ we know that there exist a point $x \in M'$ together with an arrow $y \xleftarrow{\,g\,} \phi(x)$, as the map $\phi$ is Morita. Consequently, the fact that if $\flat$ is a quasi-isomorphism at a point then it is so at every element in the orbit through such a point, allows us to guarantee that $\flat_y$ is a quasi-isomorphism.
%\begin{equation*}
%\begin{tikzcd}[row sep=2em, column sep=2em]
%0 \arrow[rr] & &A'_x \arrow[rr, "\rho'"] \arrow[dd]\arrow[dr] & &T_xM' \arrow[rr] \arrow[dd] \arrow[dr]& & 0 \arrow[rr]  \arrow[dd] \arrow[dr] & &0 &\\
%&	0 \arrow[rr]&& A_{\phi(x)} \arrow[rr, "\rho\hspace{1cm}"] \arrow[dd]& & T_{\phi(x)}M \arrow[rr] \arrow[dd, crossing over] && 0 \arrow[rr]  \arrow[dd] && 0\\
%0 \arrow[rr] && 0 \arrow[rr]  & & T^{\ast}_xM'  \arrow[rr ] & & A'^{\ast}_x \arrow[rr] & &0&\\
%&	0 \arrow[rr] && 0 \arrow[rr] \arrow[ul] && T_{\phi(x)}^*M \arrow[rr]  \arrow[ul] && A_{\phi(x)}^* \arrow[rr]  \arrow[ul] && 0.
%\end{tikzcd}
%\end{equation*}	
\end{proof}

This clearly enables us to define a notion of stacky cosymplectic structure, namely, a notion of cosymplectic structure over the differentiable stack $[M/G]$ presented by $G\rr M$.

\section{Moment map morphisms}

This section is devoted to study some aspects concerning a notion of Hamiltonian action of foliation Lie 2-groups on Lie groupoids endowed with 0-shifted cosymplectic structures. In our situation, such a study can be done by applying several key results of the general theory developed in \cite{HoffmanSjamaar}.

A \emph{Lie $2$-group} is a group internal to the category of Lie groupoids. We will assume that the reader is also familiar with the terminology described in \cite[s. 6]{HoffmanSjamaar}. Suppose that $K_1 \rightrightarrows K_0$ is a foliation Lie $2$-group with associated crossed module of Lie groups $(K,N,\partial,\alpha)$ and Lie $2$-algebra $\mathfrak{k}_1 \rightrightarrows \mathfrak{k}_0$. In this case, we have that $\textnormal{Lie}(\partial):\mathfrak{n}\to \mathfrak{k}=\mathfrak{k}_0$ is injective. As the Lie algebra $\textnormal{Lie}(\partial)(\mathfrak{n})\cong \mathfrak{n}$ is an ideal in $\mathfrak{k}$, we may consider the quotient Lie algebra $\mathfrak{k}/\mathfrak{n}$. Let $\pi:\mathfrak{k}\to \mathfrak{k}/\mathfrak{n}$ the quotient map. The pair $(\pi\circ \textnormal{Lie}(t_K),\pi):(\mathfrak{k}_1 \rightrightarrows \mathfrak{k}_0)\to (\mathfrak{k}/\mathfrak{n}\rightrightarrows \mathfrak{k}/\mathfrak{n})$ is a Morita morphism of Lie $2$-algebras and, as a consequence, the Lie 2-group action determined by the adjoint action of $K_1$ on $\mathfrak{k}_1$ descends to a Lie 2-group action of $K_1$ on $\mathfrak{k}/\mathfrak{n}$. Therefore, we can define a \emph{coadjoint} Lie 2-group action $\textnormal{Ad}^\ast:(K_1\times (\mathfrak{k}/\mathfrak{n})^\ast\rightrightarrows K_0\times (\mathfrak{k}/\mathfrak{n})^\ast)\to ((\mathfrak{k}/\mathfrak{n})^\ast \rightrightarrows (\mathfrak{k}/\mathfrak{n})^\ast)$ by dualizing the latter action. 

Let us now consider a Lie $2$-group action of $K_1\rr K_0$ on a foliation groupoid $G\rr M$. For every $\xi\in \mathfrak{k}$ the pair $((1_\xi)_{G},\xi_{M})$, which consists of the fundamental vector fields of the respective Lie group actions, gives rise to a multiplicative vector field on $G$. The \emph{fundamental vector field} of $\overline{\xi}\in \mathfrak{k}/\mathfrak{n}$ is by definition the basic vector field on $G$ associated with $((1_\xi)_{G},\xi_{M})$ for any representative $\xi\in \mathfrak{k}$, consult \cite[s. 5.4]{HoffmanSjamaar}. This is well-defined and determines a Lie algebra anti-morphism $\mathfrak{k}/\mathfrak{n}\to \Gamma_{\textnormal{bas}}(G)$.

\begin{definition}
Let $G\rr M$ be a Lie groupoid endowed with a 0-shifted cosymplectic structure $(\omega,\eta)$ and $K_1\rr K_0$ denote a foliation Lie 2-group. We say that a Lie 2-group action of $K_1\rr K_0$ on $G\rr M$ is \emph{Hamiltonian} if there exists a morphism of Lie groupoids called \emph{moment map} $\mu:G\to (\mathfrak{k}/\mathfrak{n})^\ast \cong \mathfrak{n}^\circ \subseteq \mathfrak{k}^\ast$ such that
\begin{itemize}
\item the action of $K_0$ on $(M,\omega,\eta)$ is precosymplectic,
\item for each $\overline{\xi}\in \mathfrak{k}/\mathfrak{n}$ we have that $\eta(\xi_M)=0$ and $d\mu^\xi_0=\iota_{\xi_M}\omega$, and
\item $\mu$ is $\textnormal{Ad}^\ast$-equivariant, in the sense that it is 2-equivariant with respect to the Lie 2-group action of $K_1\rr K_0$ on $G\rr M$ and the coadjoint 2-action of $K_1\rr K_0$ on $(\mathfrak{k}/\mathfrak{n})^\ast \rightrightarrows (\mathfrak{k}/\mathfrak{n})^\ast$.
\end{itemize}
\end{definition}

One can build examples of these structures by applying to Examples \ref{Ex:Precosymplecticfromcosymplectic} and \ref{Ex:CosymplecticTrivialProduct} the following general procedure.

\begin{example}
\label{Ex:precosymplecticManifoldwithFoliaiton}
Let $(M,\omega,\eta)$ be a precosymplectic manifold with corresponding foliation $\mathcal{F}_\flat$. For any source-connected foliation groupoid $G\rr M$ with Lie algebroid equal to $T\mathcal{F}_\flat$ it follows that the pair $(\omega,\eta)$ determines a 0-shifted cosymplectic structure $G$. Indeed, this is consequence of \cite[Lem. 5.3.8]{HoffmanSjamaar}, as such forms are infinitesimally invariant since they are closed and horizontal with respect to $\mathcal{F}_\flat$ because $\ker(\flat)=\ker(\omega)\cap \ker(\eta)$. In particular, one may choose $G$ to be the monodromy groupoid or holonomy groupoid of the foliation $\mathcal{F}_\flat$. Additionally, if $K$ is a Lie group acting in a Hamiltonian fashion on $(M,\omega,\eta)$ with moment map $\mu_0:M\to (\mathfrak{k}/\mathfrak{n})^\ast$ then we can promote $K$ to a foliation Lie 2-group $K_1\rr K$ acting on $G\rr M$ in a Hamiltonian manner with moment map morphism $\mu=s^\ast \mu_0$ covering $\mu_0$. Compare \cite[s. 7.2]{HoffmanSjamaar}.

In particular, under the assumptions set out in Example \ref{Ex:Complexspace}, we can consider the precosymplectic manifold $(X_0,\sum_{j=k+1}^ndz_j\wedge d\bar z_j, d\theta)$ with moment map  $\mu(z,\theta)=\sum_{j=k+1}^n(|z_j|^2-1)\xi_j$ and foliation $\mathcal{F}_\flat$ whose tangent distribution is defined by $$T_{(z_1,\ldots, z_n,\theta)}\mathcal{F}_\flat=T_{z_1}S^1\times \cdots \times T_{z_k}S^1\times 0^{n-k}\times T_\theta S^1\subset T_{z}X_0.$$

On the one hand, since the holonomies are trivial the holonomy groupoid of $\mathcal F_\flat$ is 
$$\mathbb{T}^k\times \mathbb{T}^k\times \mathbb{C}^{n-k}\times S^1\rr \mathbb{T}^k\times \mathbb{C}^{n-k}\times S^1,$$
where the structural maps come from the submersion groupoid associated with the projection $X_0\to \mathbb{C}^{n-k}$ onto the last $n-k$ coordinates. On the other hand, the inclusion $i:\mathbb{R}^k\hookrightarrow \mathbb{R}^n$ is integrated to the Lie group morphism $\partial:\mathbb{R}^k\to \mathbb{T}^n$ given by $(t_1,\ldots, t_k)\mapsto (e^{2\pi i t_1},\ldots,e^{2\pi i t_k},1,\ldots,1)$. Such a Lie group morphism defines a crossed module of Lie group with the trivial action of $\mathbb{T}^n$ on $\mathbb{R}^k$, so that we obtain a foliation Lie 2-group $\mathbb{R}^k\ltimes_\partial \mathbb{T}^n\rr\mathbb{T}^n$. It is worth mentioning that the groupoid structure on $\mathbb{R}^k\ltimes_\partial \mathbb{T}^n$ is the one obtained from the transformation groupoid with respect to the action of $\mathbb{R}^k$ on $\mathbb{T}^n$ through $\partial$.
\end{example}

We now aim to use the main results of \cite{HoffmanSjamaar} to formulate their analogues within the framework of cosymplectic geometry. Let us begin by describing a reduction procedure. Let $(G\rr M,\omega,\eta)$ be a 0-shifted cosymplectic groupoid admitting a Hamiltonian action of a foliation Lie 2-group $K_1\rr K_0$ with moment map morphism $\mu:G\to (\mathfrak{k}/\mathfrak{n})^\ast$. If $0\in (\mathfrak{k}/\mathfrak{n})^\ast$ is a regular value of $\mu_0$ then we get that $\mu_1^{-1}(0)\rr \mu_0^{-1}(0)$ gives rise to a Lie subgroupoid of $G\rr M$ which we refer to as the \emph{zero fiber} of $\mu$. Observe that such a zero fiber naturally inherits a Lie 2-group action of $K_1\rr K_0$. The fact that $\ker(\flat)=\ker(\omega)\cap \ker(\eta)$ allows us to use the foliation $\mathcal{F}_\flat$ to construct a pair $(R_1,\phi)$ where $R_1\rr R_0$ is a foliation groupoid equipped with a regular\footnote{Regular here means that the Lie 2-group action satisfies the following conditions: 1) it is \emph{locally leafwise transitive} in the sense that for all $z\in R_0$ the image of the map $\mathfrak{n}\to T_z R_0$ defined by $\xi\mapsto (\textnormal{Lie}(\partial)(\xi))_{R_0}(z)$ is $\im(\rho_{R,z})$, 2) the action of $K$ on $R_0$ is free, and 3) if $n\in N$ verifies $n\cdot p=p$ for any $p\in R_1$ then $n\in \ker(\partial)$.} Lie 2-group action of $K_1\rr K_0$, $\phi: R_1\to \mu_1^{-1}(0)$ is a 2-equivariant Morita map, and the $K_0$-orbits of $R_0$ are the leaves of the foliation $\mathcal{F}_{\phi_0^\ast \flat}$. Here $\phi^\ast \flat$ stands for the Lichnerowicz map associated with the pair $(\phi_0^\ast \omega, \phi_0^\ast \eta)$ which remains having constant rank since $\phi$ is a Morita map. Actually, $\phi^\ast \flat$ is also a quasi-isomorphism for the corresponding diagram \eqref{diag:cochain}, see Proposition \ref{Prop: Morita invariance}. Once again, such a pair $(R,\phi)$ can be constructed by mimicking the proof of \cite[Prop. 8.2.1]{HoffmanSjamaar} and considering instead the foliation $\mathcal{F}_\flat$. 

The notion of a principal bundle in the 2-category of differentiable stacks was introduced in \cite{BursztynNosedaZhu}, while its groupoid analogue was defined in \cite[s. 6.13]{HoffmanSjamaar}. Using this terminology, one can introduce a reduction procedure for $0$-shifted cosymplectic structures by following the analogous construction in the case of 0-shifted symplectic structures. 

\begin{proposition}
Suppose that $N$ acts freely on $R_1$. Then, the orbit space $R_1/N$ is a  manifold and the quotient projection $R_1\to R_1/N$ is a principal $N$-bundle. Furthermore, there exist a 0-shifted cosymplectic groupoid $(K\times_NR_1\rr R_0,\omega^{\textnormal{red}},\eta^{\textnormal{red}})$	and a Lie groupoid morphism $\psi: R_1\to  K\times_NR_1$ covering the identity $\textnormal{id}_{R_0}$ and defining a principal $K_1$-bundle such that $\psi^\ast \omega^{\textnormal{red}}=\phi^\ast \omega$ and $\psi^\ast \eta^{\textnormal{red}}=\phi^\ast \eta$.
\end{proposition}
\begin{proof}
We will only provide a sketch of the proof, as it is completely similar to that Theorem 9.1.1 in \cite{HoffmanSjamaar}. On the one hand, the first assertion in the statement follows from \cite[Lem. 9.1.3]{HoffmanSjamaar} since the Lie 2-group action of $K_1$ on $R_1$ is regular. On the other hand, the manifold of arrows of the Lie groupoid $K\times_NR_1\rr R_0$ is defined as the associated bundle obtained out of the quotient of $K\times R_1$ by the action of $N$ given by $n\cdot (k,p)=(k\partial(n^{-1}),n\cdot p)$. The groupoid structure is set by  
$$s[k,p]=s(p),\qquad t[k,p]=k\cdot t(p),\qquad m([k,p],[k',p'])=[kk',m(k'^{-1}\cdot p,p')]$$
$$ u(p_0)=[e,u(x)], \qquad [k,p]^{-1}=[k^{-1},k\cdot p^{-1}].$$

Because $K$ is connected, the forms $\omega^{\textnormal{red}}=\phi_0^\ast \omega$ and $\eta^{\textnormal{red}}=\phi_0^\ast \eta$ are $K$-invariant. Since the action $N$ preserves the $s$-fibers, both $\omega^{\textnormal{red}}$ and $\eta^{\textnormal{red}}$ become basic forms on the Lie groupoid $K\times_N R_1$. Besides, the $K_0$-orbits of $R_0$ are the leaves of the foliation $F_{\phi_0^\ast \flat}$, so that the pair $(\omega^{\textnormal{red}},\eta^{\textnormal{red}})$ defines a 0-shifted cosymplectic structure on $K\times_N R_1$.

Let us define $\psi:R_1\to K\times_N R_1$ as $\psi_0=\textnormal{id}_{R_0}$ and $\psi_1(p)=[1_K,p]$. This is the Lie groupoid morphism mentioned in the statement with the desired properties, compare \cite[Prop. 9.1.4]{HoffmanSjamaar}.
\end{proof}

As in the symplectic case, this reduction procedure yields new insights even in the setting of cosymplectic reduction by an ordinary Lie group.

\begin{corollary}
Let $K$ be a Lie group and let $(M,\omega,\eta)$ be a cosymplectic manifold equipped with a Hamiltonian $K$-action and moment map $\mu: M\to \mathfrak{k}^\ast$. If $0$ is a regular value of $\mu$ then $\mu^{-1}(0)/K$ is a cosymplectic stack. Moreover:
\begin{itemize}
\item If the action of $K$ on $\mu^{-1}(0)$ is proper then $\mu^{-1}(0)/K$ is a cosymplectic orbifold.
\item If the action of $K$ on $\mu^{-1}(0)$ is free and proper then $\mu^{-1}(0)/K$ is a cosymplectic manifold.
\end{itemize}
\end{corollary}

One can also establish a cosymplectic version of the so-called Kirwan convexity theorem in this context \cite{Kirwan}. Let us further assume that $K_1\rr K_0$ is of \emph{compact type}, the latter meaning that the Lie algebra $\mathfrak{k}_0$ is a compact Lie algebra and the orbit space $K_0/K_1$ is compact. A \emph{maximal $2$-torus} of $K_1$ is a full Lie $2$-subgroup $T_1\rr T_0$ of $K_1\rr K_0$ such that $T_0$ is connected and $\textnormal{Lie}(T_0)=\mathfrak{t}_0$ is a maximal abelian subalgebra of $\mathfrak{k}_0$. Maximal $2$-torus are also uniquely determined up to categorical conjugation \cite[s. 6.12]{HoffmanSjamaar}. It follows that the corresponding quotient Lie algebra $\mathfrak{t}/\mathfrak{a}$ is in a natural way a direct summand of $\mathfrak{k}/\mathfrak{n}$, so we can identify $(\mathfrak{t}/\mathfrak{a})^\ast$ with a subspace of $(\mathfrak{k}/\mathfrak{n})^\ast$. The \emph{Weyl group} of $K_1$ relative to $T_1$ is by definition the Weyl group of the pair $(\mathfrak{k}/\mathfrak{n},\mathfrak{t}/\mathfrak{a})$. Let us choose a closed Weyl chamber $C$ for the Weyl action on $(\mathfrak{t}/\mathfrak{a})^\ast$. We define the \emph{moment map image} by $\mu(G)=\mu_1(G)=\mu_0(M)$ and its \emph{moment body} by $\triangle(G)=\mu(G)\cap C$. 

As an immediate consequence of Proposition \ref{prop:MomentBodyMorseBott} and \cite[Prop. 7.4.1]{HoffmanSjamaar} we deduce that:

\begin{proposition}
If the action of $K_0$ on $M$ is clean then moment body $\triangle(G)$ is a closed convex polyhedron.
\end{proposition}

We conclude by noting that, under the same assumptions, the component functions of the moment map morphism define Morse–Bott Lie groupoid morphisms in the sense of \cite{OrtizValencia}. To make this assertion precise, we define a Lie groupoid morphism $F:(G\rr M)\to (\mathbb{R}\rr \mathbb{R})$ to be of \emph{Morse--Bott type} if the basic function $f:M\to \mathbb{R}$ that characterizes it is a Morse--Bott function. In this case, the critical point set $\textnormal{Crit}(f)$ is saturated, so that each of its connected components contains critical groupoid orbits. Hence, as another immediate consequence of Proposition \ref{prop:MomentBodyMorseBott} and \cite[Prop. 4.18]{OrtizValencia} we obtain that:

\begin{proposition}
If the induced action of $T_0$ on $M$ is clean then the component Lie groupoid morphism $\mu^\xi: (G\rr M)\to (\mathbb{R}\rr \mathbb{R})$ is of Morse--Bott type for all $\xi\in \mathfrak{t}/\mathfrak{a}$. Moreover, the index at every non-degenerate critical submanifold is even.
\end{proposition}

It is worth stressing that all these features can be used to provide a sort of Delzant's classification for what we would call a toric cosymplectic stack. This can be done by following the references \cite{BazzoniGoertsches2,Hoffman} in order to classify simple convex polytopes equipped with some additional combinatorial data.

\end{document}